\documentclass[11pt]{article}
\usepackage[utf8]{inputenc}
\usepackage{amsmath, amsthm, amsfonts, amssymb, mathrsfs, lmodern, graphicx, hyperref, ragged2e, listings, tikz}
\usepackage{caption, subcaption}
\usepackage{array}
\usepackage{dirtytalk}
\usepackage[sort,nocompress]{cite}
\usepackage{lipsum} 
\usepackage{enumitem}
\setitemize{itemsep=1em,topsep=1em}
\setenumerate{itemsep=1em,topsep=1em}

\usepackage{xcolor, soul}
\sethlcolor{green}
\usepackage[margin=1.2in]{geometry}
\usepackage{etoolbox}
\patchcmd{\thebibliography}{\settowidth}{\setlength{\itemsep}{1.0\baselineskip}\settowidth}{}{}

\theoremstyle{plain}
\newtheorem{theorem}{Theorem}[section]
\newtheorem{lemma}{Lemma}[section]
\newtheorem{definition}{Definition}[section]
\newtheorem{corollary}{Corollary}[section]
\newtheorem{remark}{Remark}[section]
\newtheorem{observation}{Observation}[section]
\newtheorem{note}{Note}[section]
\newtheorem{example}{Example}[section]
\newtheorem{problem}{Problem}[section]

\title{Eccentricity Spectra and Integral Eigenvalues of Zero Divisor Graphs}

\author{Gunajyoti Saharia$^{*,1}$, Sanghita Dutta$^{*,2}$, Jibitesh Dutta$^{\dagger,3}$\\
          \small $^{*}$Department of Mathematics, NEHU, Shillong, India\\
          \small$^{\dagger}$Mathematics Division, Department of Basic Sciences and Social Sciences,   
NEHU, Shillong, India\\
\\
}

\begin{document}

\maketitle
\begin{abstract}
In this work, we study the eccentricity spectra of zero divisor graphs (ZDGs) associated with the ring $\mathbb{Z}_n.$ While previous studies have examined the Laplacian and distance Laplacian spectra of ZDGs, the eccentricity spectra have remained largely unknown due to the unique features of the eccentricity matrix. More specifically, we prove that for a prime $p$, the ZDG and extended ZDG of $\mathbb{Z}_{p^t}$ have integral eccentricity eigenvalues for $t \geq 3$ and $t \geq 2$, respectively. We also find the eccentricity spectra for specific classes of ZDGs and the relationship between the eccentricity matrix of these ZDGs and their tree structures using matrix analysis tools. In addition, for the usefulness of the energy gap in applications, we have calculated the eccentricity energy gap of ZDGs. These findings reveal interesting behaviours of the eccentricity matrix and may contribute to a more profound understanding of the structural properties of ZDGs.
\end{abstract}



\section{Introduction}

\noindent The eccentricity matrix is a powerful yet underexplored tool for analyzing a graph's structural properties. For a graph $G$, this matrix stands apart from the commonly studied distance and adjacency matrices, offering unique qualities that merit further investigation. While the adjacency and distance matrices of connected graphs are known to be irreducible, the same may not hold for the eccentricity matrix \cite{compare}. Despite its potential, little work has been done on the eccentricity matrix, though it has numerous applications across various fields \cite{11}.\\

\renewcommand{\thefootnote}{}
\footnotetext{$^1$gunajyoti1997@gmail.com}
\footnotetext{$^2$sanghita22@gmail.com}
\footnotetext{$^3$jdutta29@gmail.com}

\noindent Zero divisor graphs (ZDGs), defined over the set of zero divisors of a ring,  provide a useful approach to studying algebraic structures. The relations between ZDG and the eccentricity matrix show interesting connections, primarily when concentrating on integral eigenvalues. These eigenvalues are quite interesting because they often represent stable states in physical systems and optimal solutions in control theory \cite{frequency}. Anderson and Livingston provided the modern definition of ZDG \cite{ANDERSON}. There are several versions of ZDGs, including the famous extended ZDG introduced by Bennis et al. \cite{Bennis}.\\

\noindent The spectrum and spectral radius of graph matrices are among the most studied properties, as broadly documented in research papers \cite{SR1,SR2,SR3}. The study of matrix spectra in ZDG began in the $21^{st}$ century. In $2015,$ Young \cite{young} investigated the determinant and rank of the adjacency matrix of $\Gamma(\mathbb{Z}_n)$, proposing a scheme to determine its non-zero eigenvalues. Magi et al. \cite{magi} studied the adjacency spectrum, while Chattopadhyay et al. \cite{chat} contributed important findings on the Laplacian spectrum and characterized key graph parameters. Several mathematicians made subsequent notable contributions \cite{pirzada,patil,bajaj}. Moreover, the survey \cite{survey} and monograph \cite{monograph} gathered important results. However, despite vast research on graph spectra, the eccentricity spectra of ZDGs have not been investigated in detail. This gap prompts us to study the eccentricity spectra.\\

\noindent Symmetric matrices are well-known for having real eigenvalues. While these matrices have real eigenvalues, determining the conditions under which they have integral eigenvalues is not easy. However, there are some studies that have examined the integrality of eigenvalues and highlighted their importance in applications such as connectivity and cospectrality in symmetric matrices \cite{La1, La2, La3, hanlon, von}. Although numerous matrices have been characterized under specific conditions, a general characterization of integral eigenvalues is unavailable. Current methods include results such as that a nonsingular integer matrix has integer eigenvalues if its inverse can be expressed as the sum of rank-one matrices. Previous studies focused on specific matrix classes, such as Laplacians and other graph-related matrices\cite{chat, pirzada}. However, the integrality of eigenvalues for the eccentricity matrix, specifically in ZDGs, has not been investigated. Therefore, it is imperative to investigate the conditions under which the eccentricity matrix of ZDGs exhibits integral eigenvalues.\\

\noindent Trees are a fundamental type of graph, widely studied for their simplicity and numerous practical applications \cite{diestel, cormen, felsenstein}. Hence, it is natural to investigate which types of ZDGs form trees.\\

\noindent The energy gap is a crucial metric for assessing a graph's structural properties, with larger gaps indicating stronger connectedness and better network synchronization. Therefore, we seek to explore how the concept of the energy gap can be applied to the eccentricity matrix of the ZDG of $\mathbb{Z}_n$ and its complement.\\

\noindent In this work, our primary objectives are to: (i) identify which ZDGs of $\mathbb{Z}_n$ have an integral eccentricity spectrum, (ii)to study the relationship between the eccentricity matrix of these ZDGs and their tree structures and (iii) compute the eccentricity spectrum of $\Gamma(\mathbb{Z}_n)$ for different values of $n$. In addition, for the usefulness of the energy gap in applications, we will compute the eccentricity energy gap of ZDGs.\\

\noindent The rest of the paper is organized as follows: In Section \ref{sec 2}, we introduce the necessary preliminary information and notations, while in Section \ref{Spectra}, we examine the eccentricity spectrum of ZDG of $\mathbb{Z}_n$. The relationship between ZDGs' eccentricity matrix and tree structures, along with its implications, is discussed in Section \ref{emTree}. We examined the conditions under which the eccentricity matrix has integral eigenvalues in Section \ref{Integrality}. Finally, we examine the ZDG energy gap and its complement in Section \ref{Gap}.\\

\section{Notations, Definitions, and Preliminary Results}
\label{sec 2}

In this section, we shall introduce the basic notations, definitions, and initial results that we shall use throughout this paper. We begin by introducing the essential notations and mathematical symbols that are important for understanding our work. Then, we shall define key concepts related to the spectra of the eccentricity matrix of zero divisor graphs.\\

\noindent In order to maintain coherence and clarity we adopt the following notations. A square matrix of order $n$ with real entries is denoted by $M_n(\mathbb{R}).$ Let $I_n,$ $\det(M),$ $P_M(x),$ and $\sigma(M)$ denote the identity matrix of order $n,$ determinant, characteristic polynomial, and spectrum (i.e. the set of distinct eigenvalues with their multiplicities) of the matrix $M,$ respectively. The matrix with all entries equal to one of dimension $m\times n$ is denoted by $J_{m\times n}.$ A quotient matrix, denoted by $Q,$ is a reduced matrix obtained by grouping and summarizing the rows and columns of an original matrix based on an equivalence relation.\\

\noindent For a ring $R,$ the set of zero-divisor is represented by $Z(R)$, while $Z(R)^*$ denotes the non-zero zero-divisors, i.e. $Z(R)^*=Z(R)\setminus \{0\}.$ \\

\noindent Throughout this paper we primarily consider the ring of integers modulo $n,$ denoted by $\mathbb{Z}_n,$
where $n$ is essentially composite, otherwise, the ring becomes an integral domain.\\

\noindent A graph $G$ is a pair $(V,E),$ where $V$ is a set of vertices and $E$ is a set of edges connecting pairs of vertices. Two distinct vertices $u$ and $v$ are adjacent if and only if there is an edge between them and is denoted by $u\sim v.$ A graph $G$ is said to be a simple graph if there is no repeated edges between any pair of vertices and self loops.  In this paper, we will consider only simple graphs. A null graph is the graph, where there is no edge between any pair of vertices. A path in a graph is a sequence of non-repeated vertices and edges and a path with $n$ vertices is denoted by $P_n,$ while a  complete graph $K_n$ is a graph in which every pair of distinct vertices is connected by an edge. The complement of a graph $G$ is a graph $\overline{G}$ on the same vertices, where two vertices are adjacent in $\overline{G}$ if and only if they are not adjacent in $G.$ The distance between two vertices $u$ and $v$ in $G,$ denoted by $d(u,v),$ is the length of the shortest path connecting them. While the eccentricity of a vertex $u$ in a graph $G$, denoted as $e(u)$, is defined as the maximum distance from $u$ to any other vertex, the eccentricity matrix of the graph, denoted by $\varepsilon(G) = (\epsilon_{uv})$, is given by \cite{wanganti}:
\[
\epsilon_{uv} = 
\begin{cases}
d(u,v) & \text{if } d(u,v) = \min\{e(u), e(v)\}, \\
0 & \text{otherwise}.
\end{cases}
\]

\noindent An equitable partition of the vertex set of a graph is a partition where the vertices are divided into disjoint subsets such that each vertex in a subset has the same number of neighbours in every other subset. This type of partition is useful in various graph-theoretic studies as it helps to simplify the structure of the graph while preserving certain properties.\\

\noindent The join of two graphs $G_1$ and $G_2$ is a graph formed by taking the union of the vertex sets of $G_1$ and $G_2,$ and adding edges between every vertex of $G_1$ and every vertex of $G_2.$ The generalized join extends this concept to $n$-many graphs. Let $G$ be a graph of order $n,$ where $V(G)=\{v_1, v_2, \ldots,v_n\}.$ Let $H_1, H_2, \ldots, H_n$ be $n$ pairwise disjoint graphs. The $G$-generalised graph $G\left[H_1, H_2, \ldots, H_n\right]$ of $H_1, H_2, \ldots, H_n$ is the graph formed by replacing each vertex $v_i$ in $G$ by the graph $H_i$ and connecting every vertex in $H_i$ to every vertex in $H_j$ whenever $v_i \sim v_j$ in $G$.\\

\noindent Recall that the spectrum of a matrix $M$ refers to the set of its eigenvalues along with their multiplicities. Two important characteristics derived from the spectrum are the spectral radius and the energy of the matrix. The spectral radius, denoted by $\rho(M),$ is defined as the absolute value of the largest eigenvalue of $M,$ i.e. if $\lambda_1, \lambda_2, \ldots, \lambda_n$ are the eigenvalues of $M,$ then $\rho(M)=$ max$\{|\lambda_1|, |\lambda_2|, \ldots, |\lambda_n|\}$. On the other hand, the energy of the matrix, denoted by $\mathscr{E}(M),$ is defined as the sum of the absolute values of all its eigenvalues. Further,the absolute energy gap between two graphs $G$ and $H$ is defined as the difference between the energies of matrices connected with these graphs.\\
 
\noindent The \textit{zero-divisor graph} of a ring $R$, denoted by $\Gamma(R)$, is defined on the set of non-zero zero-divisors, $Z(R)^*$. In this graph, two distinct vertices $u$ and $v$ are adjacent if and only if $uv = 0$. The \textit{extended zero-divisor graph}, denoted by $\Gamma_E(R)$ is defined with the vertex set $Z(R)^*$ and two distinct vertices $u$ and $v$ are adjacent if and only if $u^mv^n=0,$ for some positive integers $m$ and $n.$ It is worth mentioning that for $m=1=n,$ the \textit{extended} ZDG becomes ZDG.\\

\noindent Let $x\in R.$ Then the annihilator of $x$ is the set $ann_R(x)=\{y \in R \hspace{5pt} | \hspace{2pt} yx=0\}.$ 
The concept of annihilators of an element in a ring helps understanding the structure of zero divisors in the ring.\\

\noindent The \textit{compressed zero divisor graph} is defined on the equivalence classes of zero divisors (i.e. each vertex is an annihilator set) of a ring $R.$ This graph simplifies the zero divisor graph by grouping vertices that share equivalent properties. This graph can be precisely defined as:

\begin{definition}\cite[Definition 1.1]{comzero}
    Let $R$ be a ring and $r \in R$. Let $[r]_{R}=\{s \in R \hspace{5pt}
 \lvert \hspace{5pt} Ann_{R}(r)=Ann_{R}(s)\}$ and $R_{E}=\{[r]_{R} \hspace{5pt} \rvert \hspace{5pt} r \in R\}.$ Then, the compressed zero divisor graph $\Gamma_C(R)$ is defined as the graph $\Gamma(R_E).$ 
\end{definition}

\noindent Intuitively, in the graph $\Gamma_C(R),$ each vertex represents a group of nonzero zero-divisors of the ring $R$ that share the same annihilator. Two vertices are connected by an edge if and only if there exists a nonzero zero-divisor from the first vertex's equivalence class that annihilates a nonzero zero-divisor from the second vertex's equivalence class. Thus $\Gamma_C(R)$ organizes these zero-divisors into clusters based on their annihilators.\\

\noindent \textbf{Decomposition of zero divisor graph}\\

\noindent A graph may have an extremely intricate structure, so breaking it down into smaller, easier-to-manage parts is crucial. This also holds for zero divisor graphs, which frequently have complex structures. A zero divisor graph can be more easily analyzed and comprehended by breaking it into more manageable, known subgraphs.\\

\noindent In any finite commutative ring, every element is categorized either as a zero divisor or as a unit. A unit in a ring is an element that has a multiplicative inverse within the ring and, consequently, cannot be a zero divisor. Therefore, the number of non-zero zero divisors in $\mathbb{Z}_n$ is given by $n - \phi(n) - 1$, where $\phi(n)$ is Euler's totient function. $\phi(n)$ counts the number of elements coprime to $n$ in $\mathbb{Z}_n$, and these coprime elements cannot be zero divisors because they have no common factors with $n$ that would make their product zero modulo $n$.\\

\noindent An integer $d$ is called a proper divisor of $n$ if $d \mid n$, where $1 < d < n$. Consider $n = \prod\limits_{i=1}^r p_i^{\alpha_i},$ where $p_i$'s are primes and $\alpha_i\in \mathbb{N}$. It is easy to see that the number of proper divisors of $n$ is $s(n) = \prod\limits_{i=1}^r (\alpha_i + 1) - 2$. Define the set $\mathscr{A}(d) = \{ k \in \mathbb{Z}_n : \gcd(k, n) = d \}$. Clearly, $\{ \mathscr{A}(d_1), \mathscr{A}(d_2), \ldots, \mathscr{A}(d_{s(n)}) \}$ is a collection of pairwise disjoint sets and provides an equitable partition for the vertex set of $\Gamma(\mathbb{Z}_n)$. If $d_i \mid n$, then $\left|\mathscr{A}(d_i)\right| = \phi(\frac{n}{d_i})$ for $1 \leq i \leq s(n)$ \cite[Proposition 2.1]{young}. It is well known that the subgraphs of $\Gamma(\mathbb{Z}_n)$ induced by $\mathscr{A}(d_i)$ are $K_{\phi(\frac{n}{d_i})}$ and $\overline{K}_{\phi(\frac{n}{d_i})}$ depending on whether $n \mid d_i^2$ or not, respectively \cite[Corollary 2.5]{chat}.\\

\noindent Let $\Upsilon_n$ be a simple graph with vertex set $\{d_1, d_2, \ldots, d_k\},$ where $d_i$ is a proper divisor of $n$ for $1\leq i \leq k.$ There exists an edge $d_id_j$ in $\Upsilon_n$ if and only if $n \mid d_i d_j$, where $d_i$ and $d_j$ are proper divisors of $n.$ Thus, the zero divisor graph $\Gamma(\mathbb{Z}_n)$ can be  precisely expressed as a generalized join $\Gamma(\mathbb{Z}_n) = \Upsilon_n\left[\Gamma(\mathscr{A}(d_1)), \Gamma(\mathscr{A}(d_2)), \ldots, \Gamma(\mathscr{A}(d_{s(n)}))\right]$ \cite[Lemma 2.7]{chat}. This construction illustrates how $\Gamma(\mathbb{Z}_n)$ organizes into clusters based on the divisibility relationships among the proper divisors of $n$,  reflecting the underlying structure of the ring $\mathbb{Z}_n$.
The following example shows how the ZDG is decomposed into simpler known graphs:\\

\begin{example}
Consider the ZDG $\Gamma(\mathbb{Z}_{35})$ with $p_1=5$ and $p_2=7.$ Then the \textit{ZDG} is denoted by $\Gamma(\mathbb{Z}_{35}).$ The number of zero divisors of the ring $\mathbb{Z}_{35}$ is $10.$ The set of zero-divisors is $Z(\mathbb{Z}_{35})=\{5,7,10,14,15,20,21,25,28,30\}.$ The vertex set can be partitioned as $Z(\mathbb{Z}_{35})=\{5,10,15,20,25,30\} \bigcup \{7,14,21,28\},$ where $\{5,10,15,20,25,30\}=\mathscr{A}(5)$ and $\{7,14,21,28\}=\mathscr{A}(7).$ Since the only proper divisors of $35$ are $5$ and $7,$ the graph $\Upsilon_{35}$ isomorphic to the path graph $P_2$ on two vertices. Hence $\Gamma(\mathbb{Z}_{35})$ can be expressed as the generalized join in the following way:
    \begin{align*}
        \Gamma(\mathbb{Z}_{35})&=P_2\left[\Gamma(\mathscr{A}(5)), \Gamma(\mathscr{A}(7))\right],\\
        &=P_2\left[\overline{K}_{4}, \overline{K}_6\right].
    \end{align*}
Since $35\nmid 5^2$ and $35\nmid 7^2,$ so $\Gamma(\mathscr{A}(5))$ and $\Gamma(\mathscr{A}(7))$ are $\overline{K}_{4}$ and $\overline{K}_6,$ respectively. Here $\overline{K}_i$ denotes null graph on $i$ vertices.\\

\begin{center}
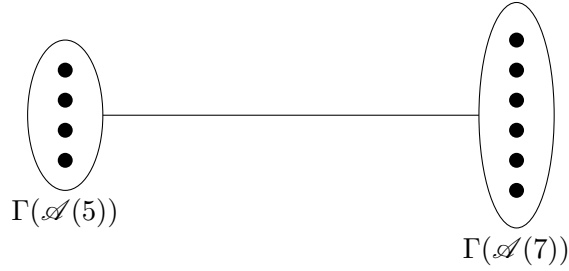

\begin{tikzpicture}
    % First ellipse
    \draw (0,0) ellipse (0.5cm and 1cm);
    \node at (0, -1.3) {$\Gamma(\mathscr{A}(5))$};

    \foreach \y in {-0.6,-0.2,0.2,0.6}
    {
        \node at (0, \y) [circle, fill, inner sep=2pt] {};
    }

    \draw (6,0) ellipse (0.5cm and 1.5cm);
    \node at (6, -1.8) {$\Gamma(\mathscr{A}(7))$};

    \foreach \y in {-1.0,-0.6,-0.2,0.2,0.6,1.0}
    {
        \node at (6, \y) [circle, fill, inner sep=2pt] {};
    }
    \draw (0.5,0) -- (5.5,0);
\end{tikzpicture}
\captionsetup{labelfont={bf},textfont={it}}
\captionof{figure}{The generalized join of $\Gamma(\mathscr{A}(5))$ and $\Gamma(\mathscr{A}(7))$ which is isomorphic to $\mathbb{Z}_{35}.$}
\end{center}   
\end{example}

\noindent  This example illustrates how breaking down a zero divisor graph into simpler, known graphs simplifies the problem. Complex mathematical problems are often tackled by decomposing them into more manageable components. For ZDG, this approach is shown by expressing $\Gamma(\mathbb{Z}_n)$ as:
$$\Gamma(\mathbb{Z}_n)=\Upsilon_n\left[\Gamma(\mathscr{A}_1), \Gamma(\mathscr{A}_2), \ldots, \Gamma(\mathscr{A}_n)\right].$$

\noindent \textbf{Matrix analysis tools:}\\

\noindent In this subsection we focus on matrix analysis tools essential for finding the spectrum of the eccentricity matrix.\\

\noindent A block matrix simplifies the process of finding the determinant of a large matrix by allowing the use of the Schur complement. In particular, for a block matrix \( M = \begin{bmatrix}
A & B \\
C & D
\end{bmatrix} \), such that its blocks $A$ and $D$ are square matrices and where \( A \) is invertible, the Schur complement of the block \( A \) in $M$ is  defined as $M/A:=D-CA^{-1}B$, captures the relationship between the blocks \( A \) and \( D \) after accounting for the influence of \( B \) and \( C \).\\

\noindent The following lemma provides an efficient method for computing determinants of block matrices using the Schur complement:

\begin{lemma}\cite[p. 4]{bapat}
    Let \( M = \begin{bmatrix}
    A & B \\
    C & D
    \end{bmatrix} \), where \( A, B, C \), and \( D \) are block matrices with \( A \) being invertible. Then 
    \[ \det(M) = \det(A) \cdot \det(D - CA^{-1}B). \]
    \label{Schur}
\end{lemma}
\noindent In matrix analysis, understanding the coronel of a matrix $M$ provides valuable insights into the spectral properties of $M$, particularly in relation to eigenvalue distributions and their sums. It is defined as follows:

\begin{definition}
    Let $M \in M_n(\mathbb{R})$ and $\textbf{1}_n$ be the column vector of one's. The sum of the elements of the matrix $(xI-M)^{-1}$ is called the coronel of the matrix $M$ and is denoted by $\Gamma_M(x)$. It can be defined as
    $$\Gamma_M(x)=\textbf{1}_n^T (xI_n-M)^{-1} \textbf{1}_n.$$
    \label{coronel}
\end{definition}
\noindent The following remark gives the coronel of a  matrix $M$ whose row sum is constant.

\begin{remark}
    Let $M\in M_n(\mathbb{R})$ such that row sum is constant and is equal to $\alpha$. Then the coronel of the matrix $M$ is given by $$\Gamma_M(x)=\frac{n}{x-\alpha}.$$
    \label{constant coronel}
\end{remark}

\noindent The following lemma highlights the utility of coronel in spectral computations and its role in evaluating determinants under specific matrix structures:

\begin{lemma}
    Let $A, J\in M_n(\mathbb{R})$. All the entries of $J$ are one. Then, 
    \begin{align*}
       \det(xI_n-A-\beta J_{n \times n})=(1-\beta \Gamma_A(x))\det(xI_n-A),\hspace{5mm} \rm{where}~~ \beta \in \mathbb{R}.
     \end{align*}
     \label{det}
\end{lemma}

\noindent In Matrix analysis, a rank-one perturbation refers to modifying a matrix $M$
by adding a matrix of rank one, typically represented as $xy^T$, where $x$ and $y$ are column vectors. The Cauchy rank-one perturbation theorem helps us understand how the eigenvalues of a matrix change when we make a small, rank-one adjustment.  This technique is crucial in analyzing how small adjustments impact the eigenvalues and determinants of the matrix. The following lemma  gives a formula to calculate the determinant of a rank-one perturbation.\\

\begin{lemma}[Determinant of Rank-One Perturbation]\cite[p. 26]{bapat}
    Let \( M \) be a square matrix of order \( n \), and let \( x \) and \( y \) be column vectors in \( \mathbb{R}^n \). A rank-one perturbation of \( M \) is given by \( M + xy^T \). The determinant of this perturbed matrix can be computed as:
    \[
    \det(M + xy^T) = \det(M) + y^T (\text{adj}(M)) x,
    \]
    where \(\text{adj}(M)\) is the adjugate (or adjoint) of the matrix \( M \).
\end{lemma}

\noindent Having established the necessary preliminaries and definitions, we now investigate the spectrum of zero divisor graphs of $\mathbb{Z}_n$ in the following section.\\

\section{Eccentricity spectrum of the zero divisor graph of $\mathbb{Z}_n$}
\label{Spectra}

Cardoso et al. deduced a general method to determine the Laplacian spectra of graphs obtained by a generalized join operation on families of graphs as mentioned in \cite[Theorem 8]{cardoso}. They specifically constructed a symmetric matrix $C_{\hat{\alpha}}(\hat{\rho})$, where $\hat{\alpha}$ is a $k$-tuple and $\hat{\rho}$ is a $\frac{k(k-1)}{2}$-tuple, to represent the Laplacian matrix of the generalized join of the underlying graphs.\\

\noindent Similarly, in \cite[Corollary 2]{cardosodistance}, the authors provided a general formula for the spectra of distance signless Laplacian matrix of generalized join operation. Using these general results, Chattopadhya et al. \cite{chat} and Pirzada et al. \cite{pirzada} subsequently investigated the Laplacian and distance Laplacian spectra of $\Gamma(\mathbb{Z}_n).$ However, Cardoso's approach cannot be directly applied to the eccentricity spectra of $\Gamma(\mathbb{Z}_n)$ due to the strange actions of the eccentricity matrix. Hence, in this section, we use a different technique (i.e. some advanced matrix analysis tools) to obtain the eccentricity spectra of $\Gamma(\mathbb{Z}_n)$.\\

\subsection{Eccentricity spectra of zero divisor graphs}
\begin{theorem}
    Let $p_1$ and $p_2$ be two distinct primes. The eccentricity spectrum of $\Gamma(\mathbb{Z}_{p_1p_2})$ is given by :$$\sigma(\varepsilon(\Gamma(\mathbb{Z}_{p_1p_2})))=\begin{Bmatrix}
-2 & 2p_1-4 & 2p_2-4\\
p_1+p_2-4 & 1 & 1
\end{Bmatrix}.$$
\label{FirstTh}
\end{theorem}

\begin{proof}

    Let $G=\Gamma(\mathbb{Z}_{p_1p_2})$, where $p_1$ and $p_2$ are the distinct proper divisors of $p_1p_2$. The number of non-zero zero divisors of $G$ is $p_1p_2-\phi(p_1p_2)-1=p_1+p_2-2$. By assigning labels to these vertices, we have the following sets:
    \begin{align*}
       \mathscr{A}(p_1)=\{r_1p_1 : r_1=1,2,\ldots,(p_2-1)\}, \\
       \mathscr{A}(p_2)=\{r_2p_2 : r_2=1,2,\ldots,(p_1-1)\}.
    \end{align*}
Clearly, $\mathscr{A}(p_1)$ and $\mathscr{A}(p_2)$ form an equitable partition of the vertex set $V(G)$. The cardinalities of $\mathscr{A}(p_1)$ and $\mathscr{A}(p_2)$ are $p_2-1$ and $p_1-1,$ respectively. When labeling the vertices, arrange those from the set $\mathscr{A}(p_1)$ first, followed by the vertices from the set $\mathscr{A}(p_2)$. Note that the sets $\mathscr{A}(p_1)$ and $\mathscr{A}(p_2)$ induce null graphs of order $p_2-1$ and $p_1-1,$ respectively. Thus,
    
    $$G=P_2\left[\overline K_{p_1-1},\overline K_{p_2-1}\right].$$

\noindent We note that, in graph $G$,  the distance between any two vertices within either $\mathscr{A}(p_1)$ or $\mathscr{A}(p_2)$ is $2$, and the distance between a vertex in $\mathscr{A}(p_1)$ and a vertex in $\mathscr{A}(p_2)$ is $1$. Therefore the eccentricity matrix of $G$ is:
    $$
        \varepsilon(G)=\begin{bmatrix}
            2(J-I)_{(p_2-1)\times (p_2-1)} & O_{(p_2-1)\times(p_1-1)}\\
            O_{(p_1-1)\times(p_2-1)} & 2(J-I)_{(p_1-1)\times (p_1-1)}
        \end{bmatrix}=\begin{bmatrix}
            M_{11} & O\\
            O & M_{22}
        \end{bmatrix},
    $$
    where $M_{11}=2(J-I)_{(p_2-1)\times (p_2-1)}$ and $M_{22}=2(J-I)_{(p_1-1)\times (p_1-1)}$.\\

\noindent Using the property of block matrices, the characteristic polynomial of $\varepsilon(G)$ is \begin{align}
        P_{\varepsilon(G)}(x)=\det(xI-M_{11})\cdot \det(xI-M_{22}).
        \label{Poly}
    \end{align}
    The above polynomial also can be obtained by using the lemma \ref{Schur}.\\
    
    \noindent Now, since the matrix $-2I_{(p_2-1)\times (p_2-1)}$ has constant row sum which is equal to $-2,$ therefore using remark \ref{constant coronel} the coronel of this matrix is 
    \begin{align*}
        \Gamma_{(-2I)}(x)=\frac{p_2-1}{x+2}.
    \end{align*}
    The matrix $xI-M_{11}$ can be expressed as $xI-2J+2I.$ Clearly, $\det(xI+2I)_{(p_2-1)\times (p_2-1)}=(x+2)^{p_2-1}.$ Now applying lemma \ref{det}, we get
    \begin{align*}
        \det(xI-M_{11}) &= \det(xI-2J+2I),\\
        &= \left[1-2\Gamma_{(-2I)}(x)\right]\cdot \det(xI+2I),\\
        &=\frac{x-2(p_2-2)}{x+2}\cdot (x+2)^{p_2-1},\\
        &= (x+2)^{p_2-2}[x-2(p_2-2)].
    \end{align*}
Similarly, \begin{align*}
        \det(xI-M_{22})= (x+2)^{p_1-2}\left[x-2(p_1-2)\right].
    \end{align*}
    From the equation \ref{Poly} the characteristic polynomial can be written as:
    \begin{align*}
        P_{\epsilon(G)}(x)= (x+2)^{p_1+p_2-4}(x-2p_1+4)(x-2p_2+4). 
    \end{align*}
    Hence, the eccentricity spectrum of the graph $G=\Gamma(\mathbb{Z}_{p_1p_2})$ is 
$$\sigma(\varepsilon(\Gamma(\mathbb{Z}_{p_1p_2})))=\begin{Bmatrix}
-2 & 2p_1-4 & 2p_2-4\\
p_1+p_2-4 & 1 & 1
\end{Bmatrix}.$$
\end{proof}

\begin{observation}
    By the Perron-Frobenius theorem\cite{perron}, the largest eccentricity eigenvalue of the ZDG of $\mathbb{Z}_{p_1p_2}$ is positive with algebraic multiplicity one.
\end{observation}

\begin{observation}
     For two distinct prime numbers $p_1$ and $p_2$ with $p_1<p_2$, the eccentricity spectral radius (i.e. the largest absolute value of the eigenvalues) of $\Gamma(\mathbb{Z}_{p_1p_2})$ is $2(p_2-2).$
\end{observation}
   
\begin{remark}
    Similar to the eccentricity spectral radius, the signless Laplacian spectral radius of $\Gamma(\mathbb{Z}_{p_1p_2})$ is also integral, specifically equal to $p_1 + p_2 - 2$.
\end{remark}

\begin{theorem}
\label{Theorem 3.2}
    For any prime $p$ ($\neq 2$) the eccentricity spectrum of $\Gamma(\mathbb{Z}_{p^3})$ is given by $$\sigma(\varepsilon(\Gamma(\mathbb{Z}_{p^3})))=\begin{Bmatrix}
-1 & -2 & 2p^2-2p-2 & \frac{p^3-4p^2+p+4}{2p^2-2p-2}\\
p-2 & p^2-p-1 & 1 & 1
\end{Bmatrix}.$$
\label{prime}
\end{theorem}

\begin{proof}
    Let $G=\Gamma(\mathbb{Z}_{p^3})$. We know $p$ and $p^2$ are the proper divisors of $p^3.$ So, the number of non-zero zero divisor of $\mathbb{Z}_{p^3}$ is $p^2-1$. Similar to the previous proof, the vertex set can be partitioned equitably into two subsets:
    \begin{align*}
        \mathscr{A}(p)=\{r_1p: r_1=1,2,\ldots, (p^2-1), \hspace{5pt} \rm{where} \hspace{5pt} p \nmid r_1\},\\
        \mathscr{A}(p^2)=\{r_2p: r_2=1,2,\ldots, (p-1), \hspace{5pt} \rm{where} \hspace{5pt} p \nmid r_2\}.
    \end{align*}
    The sets $\mathscr{A}(p)$ and $\mathscr{A}(p^2)$ induce a null graph of order $p(p-1)$ and a complete graph of order $p-1,$ respectively. So,
    $$G=P_2\left[\overline K_{p(p-1)}, K_{p-1}\right].$$
    The eccentricity matrix of this graph is given by 
    $$\varepsilon(G)=\begin{bmatrix}
        2(J-I)_{p(p-1)\times p(p-1)} & J_{p(p-1)\times (p-1)}\\
        J_{(p-1)\times p(p-1)} & (J-I)_{(p-1)\times(p-1)}
    \end{bmatrix}=\begin{bmatrix}
        A & J\\
        J^T & B
    \end{bmatrix},$$
    where $A=2(J-I)_{p(p-1)\times p(p-1)}$, \hspace{10pt} $B=(J-I)_{(p-1)\times(p-1)}$ and $J=J_{p(p-1)\times(p-1)}$.\\

\noindent Using lemma \ref{Schur}, the characteristic polynomial of $\varepsilon(G)$ is given by
    \begin{align}
        P_{\varepsilon(G)}(x)=P_A(x)\cdot P_C(x),
        \label{Poly2}
    \end{align}
    where $C=B-J^TA^{-1}J.$\\
    
    \noindent By using lemma \ref{constant coronel} and lemma \ref{det} we get
    \begin{align*}
        P_A(x)=\det(xI-A) &= \det\left(xI-2J+2I\right),\\
        &= (x+2)^{p(p-1)-1}\left[x-(2p^2-2p-2)\right].
    \end{align*}
    Similarly, \begin{align*}
        P_C(x)=\det\left(xI-C\right).
    \end{align*}
Now, $J^TA^{-1}J$ is given by\\
\begin{align*}
    J^TA^{-1}J=\frac{p(p-1)}{2p(p-1)-2}J_{(p-1)\times (p-1)}.
\end{align*}
Hence, the matrix $C$ is \begin{align*}
    B-J^TA^{-1}J=\begin{bmatrix}
    -\frac{p(p-1)}{2p(p-1)-2} & 1- \frac{p(p-1)}{2p(p-1)-2} & \cdots &  1-\frac{p(p-1)}{2p(p-1)-2}\\
     1-\frac{p(p-1)}{2p(p-1)-2} &  -\frac{p(p-1)}{2p(p-1)-2} & \cdots & 1- \frac{p(p-1)}{2p(p-1)-2}\\
     \vdots & \vdots & \ddots & \vdots\\
     1- \frac{p(p-1)}{2p(p-1)-2} & 1- \frac{p(p-1)}{2p(p-1)-2} & \cdots &  -\frac{p(p-1)}{2p(p-1)-2}
     \end{bmatrix}.
    \end{align*}
    
\noindent Therefore, $P_C(x)=(x+1)^{p-2}\left(x-\frac{p^3-4p^2+p+4}{2p^2-2p-2}\right).$\\

\noindent From the equation (\ref{Poly2}) the characteristic polynomial can be written as
\begin{align*}
        P_{\varepsilon(G)}(x)= (x+1)^{p-2}(x+2)^{p(p-1)-1}(x-2p^2+2p+2)\left(x-\frac{p^3-4p^2+p+4}{2p^2-2p-2}\right). 
    \end{align*}
    Hence, the eccentricity spectrum of the graph $G=\Gamma(\mathbb{Z}_{p^3})$ is $$\sigma(\varepsilon(\Gamma(\mathbb{Z}_{p^3})))=\begin{Bmatrix}
-1 & -2 & 2p^2-2p-2 & \frac{p^3-4p^2+p+4}{2p^2-2p-2}\\
p-2 & p^2-p-1 & 1 & 1
\end{Bmatrix}.$$
\end{proof}

\begin{remark}
The above theorem does not hold for \(p = 2\). For a detailed explanation, refer to Example \ref{Eg1} in the appendix.
\end{remark}

\noindent The following corollary gives the spectral radius of $\Gamma(\mathbb{Z}_{p^3}).$
\begin{corollary}
    For a prime $p (\neq 2)$ the eccentricity spectral radius of $\Gamma(\mathbb{Z}_{p^3})$ is $2(p^2-p-1)$.
\end{corollary}

\begin{remark}
    The above corollary clearly states that the eccentricity spectral radius of $\Gamma(\mathbb{Z}_{p^3})$ is an integer, whereas the signless Laplacian spectral radius of the same graph is not an integer \cite{signless}.
\end{remark}

\begin{theorem}
\label{Theorem 3.3}
   For a prime $p$, the eccentricity spectrum of the graph $\Gamma(\mathbb{Z}_{p^4})$ is given by
   \begin{align*}
   \sigma(\varepsilon(\Gamma(\mathbb{Z}_{p^4})))=&\begin{Bmatrix}
   -2 & 0 & -1-p-p^3-\Lambda & -1-p-p^3+ \Lambda\\ p^2(p-1) & p^2-1 & 1 & 1 
   \end{Bmatrix},
\end{align*}

\noindent \rm{where} $\Lambda=\sqrt{2+2p+p^2+2p^3-10p^4+4p^5+p^6}.$

\end{theorem}
\begin{proof}

As before,  the graph \(G = \Gamma(\mathbb{Z}_{p^4})\) can be decomposed as 
    $$
    G = P_3 \left[\overline{K}_{p^2(p-1)}, K_{(p-1)}, K_{p(p-1)}\right].
    $$
    
    The eccentricity matrix \(\varepsilon(G)\) of this graph is given by:
    \begin{align*}
    \varepsilon(G) &= \begin{bmatrix}
    2(J-I)_{p^2(p-1)\times p^2(p-1)} & O_{p^2(p-1) \times (p-1)} & 2J_{p^2(p-1)\times p(p-1)}\\
    O_{(p-1)\times p^2(p-1)} & O_{(p-1)\times (p-1)} & O_{(p-1)\times p(p-1)}\\
    2J_{p(p-1)\times p^2(p-1)} & O_{p(p-1)\times (p-1)} & O_{p(p-1)\times p(p-1)}
    \end{bmatrix}.
    \end{align*}

\noindent Similar to the above two theorems, by using matrix analysis tools the characteristic polynomial can be obtained. The detail analysis along with eigenvalue computations are given in the Appendix (see Appendix \ref{appendix:th A.3}).\\

\noindent From these computations, we conclude that the eccentricity spectrum of \(\Gamma(\mathbb{Z}_{p^4})\) is as stated in the theorem.
\end{proof}

\noindent As an immediate consequence, we have the following corollary.

\begin{corollary}
    For a prime \(p \neq 2\), the eccentricity spectral radius of $$\Gamma(\mathbb{Z}_{p^4})~~~ \text{is}~~ -1-p-p^3+\sqrt{2+2p+p^2+2p^3-10p^4+4p^5+p^6}.$$
\end{corollary}

\noindent The zero divisor graph of $\mathbb{Z}_{p^4}$ also shows contrasting result in terms of spectral radius. The following remark formally states the result.

\begin{remark}
    In contrast to the eccentricity spectral radius of $\Gamma(\mathbb{Z}_{p^4}),$ the signless Laplacian spectral radius of $\Gamma(\mathbb{Z}_{p^4})$ is an integer.  Specifically, it is equal to $(p^3-3)$ \cite{signless}.
\end{remark}

\begin{theorem} 
\label{Theorem 3.4}
For two distinct primes $p_1$ and $p_2$, the eccentricity spectrum of $G=\Gamma(\mathbb{Z}_{p_1^2p_2})$ is: 
\begin{align*}
\sigma(\varepsilon(G))= \begin{Bmatrix}
     0 & -2 & 2p_2-6 & 2(p_1-1)(p_2-1)-4\\
     p_1^2+2p_1-4 & p_1(p_2-1) & 1 & 1
    \end{Bmatrix}\bigcup \hspace{5pt} \Theta_{P_{\varepsilon(G)}(x)},
\end{align*}
\end{theorem}
\noindent where $\Theta_{P_{\varepsilon(G)}(x)}$ is the set of distinct roots of the polynomial (\ref{A})  with their multiplicities.

\begin{proof} 
Let $G=\Gamma(\mathbb{Z}_{p_1^2p_2})$, where $p_1$ and $p_2$ are two distinct primes. Similar to previous proofs, the graph $G$ can be decomposed as $$G=P_4\left[\overline K_{(p_1^2-1)}, \overline K_{(p_2-1)}, \overline K_{\phi(p_1p_2)}, K_{(p_1-1)}\right].$$
The eccentricity matrix of the graph can be realized as 
\begin{align*}
    \varepsilon(G) &= \begin{bmatrix}
    O & O & O & 3J\\
    O & 2(J-I) & O & 2J\\
    O & O & 2(J-I) & O\\
    3J & 2J & O & O
    \end{bmatrix}=\begin{bmatrix}
        O & O & O & M_{14}\\
        O & M_{22} & O & M_{24}\\
        O & O & M_{33} & O\\
        M_{41} & M_{42} & O & O    \end{bmatrix}.
\end{align*}
So, the characteristic polynomial of the matrix is 
$$P_{\varepsilon(G)}(x)=\det(xI-\varepsilon(G)).$$
As before, the detail computation of the characteristic polynomials and the eigenvalues can be found in the Appendix (see Appendix \ref{Appendix A.4}).\\

\noindent Hence, the eccentricity spectrum of $G=\Gamma(\mathbb{Z}_{p_1^2p_2})$  is as stated in the theorem. 
\end{proof} 

\noindent The following corollary provides the eccentricity spectral radius of $\Gamma(\mathbb{Z}_{p_1^2p_2}).$
\begin{corollary}
    For two distinct primes $p_1$ and  $p_2$, the eccentricity spectral radius of $\Gamma(\mathbb{Z}_{p_1^2p_2})$ is $2(p_1p_2-p_1-p_2-1).$
\end{corollary}

\begin{remark}

    The methodologies we already used to obtain the the eccentricity spectra of zero divisor graphs, can be extended to learn the spectra of eccentricity matrices associated to extended ZDG. For an application in  extended ZDG, please see example \ref{Eg2}.
\end{remark}

\noindent Thus,  in this section, we see that in contrast to the more predictable patterns followed in the adjacency and Laplacian spectra,  the eccentricity spectra of ZDGs are greatly influenced by the underlying ring's structure, making them difficult to generalize. While previous investigations, such as those by Chattopadhya et al. \cite{chat} and Pirzada et al. \cite{pirzada}, have revealed that the spectra of the Laplacian and distance Laplacian matrices of $\Gamma(\mathbb{Z}_n)$ stick to a typical pattern, the eccentricity matrix spectrum differs from this trend. In what follows, we pose an open problem for further investigation.

\begin{problem}
    For a given $n\in \mathbb{N},$ what is the eccentricity spectrum of $\Gamma(\mathbb{Z}_{p^n})?$
\end{problem}

\section{Relation Between the Eccentricity Matrix of Zero Divisor Graphs and Tree Structures}
\label{emTree}

\noindent By examining zero divisor graphs, we find that $\Gamma(\mathbb{Z}_n)$ forms a tree for specific values of $n$. The eccentricity matrix of a tree has unique features essential for understanding its structure. The articles \cite{wang, mahato, he} extensively cover this matrix, highlighting its distinct structural properties.\\

\noindent For any left Artinian ring $R$ with at least four vertices, $\Gamma(R)$ is a star graph if and only if $R$ is isomorphic to the direct product of a division ring and a ring of order $2$ \cite{arti}. Motivated by the relation between ZDG and tree, here in this section we investigate the relationship between a ZDG and its eccentricity matrix.\\

\noindent The following two theorems shed light on the eccentricity matrix of a tree.

\begin{theorem}(\cite{wang}\cite[Theorem 3.2]{mahato})
    Let $T$ be a tree of order $n,$ other than $P_2,$ and $\lambda_n(T )$ be the least eigenvalue of the eccentricity matrix of $T.$ Then, $\lambda_n(T ) \leq -2$ with equality if and only if $T$ is the star.
    \label{least}
\end{theorem}
\noindent This theorem guarantees that there is no tree whose least eigenvalue of the eccentricity matrix is greater than $-2.$ Moreover, if $T$ is a tree which is not a star then the least eigenvalue of the eccentricity matrix is less than $-2\sqrt{2}$ \cite{solution}.

\begin{theorem}(\cite{wang}\cite[Theorem 3.3]{mahato})
    The eccentricity matrix of a tree is irreducible.
\label{irreducible}
\end{theorem}
 
 \noindent The above lemma \ref{irreducible} shows that the eccentricity matrix of a graph can be irreducible. There are graphs for which the eccentricity matrix is reducible \cite{reducible}, contrasting with the fact that the adjacency matrix of any connected graph is always irreducible\cite{adja}.\\

 \noindent We now provide the condition on $n$ for which the eccentricity matrix of the zero divisor graph of $\mathbb{Z}_n$ is a tree graph.

 \begin{theorem}
     The zero divisor graph of $\mathbb{Z}_n$ is a tree if and only if $n=2p,$ where $p$ is a prime. In particular, $\Gamma(\mathbb{Z}_{2p})$ is a star graph.
     \label{tree}
 \end{theorem}

\begin{proof}
A positive integer 
$n$ can always be categorized into one of the following four cases:\\

\noindent \textbf{Case I:} If $n=\prod\limits_{i=1}^kp_i^{\alpha_i},$ where $\alpha_i\geq 2$ for some $i,$ then we will show that $\Gamma(\mathbb{Z}_n)$ is not a tree.\\
\noindent Let $d=\prod \limits_{i=1}^kp_i$ be a proper divisor of $n.$ Define the following set:
\begin{align*}
    \mathscr{A}(d)=\{rd \hspace{2pt}| \hspace{2pt} r\in \mathbb{Z}_{\frac{n}{d}}, \hspace{2pt} gcd\left(\frac{n}{d},r\right)=1\}.
\end{align*}
Then the ZDG induced by $\mathscr{A}(d)$ is a complete graph on $\left(\frac{n}{d}-1\right)$ vertices.\\
\noindent Now consider the subgraph $H$ of $\Gamma(\mathbb{Z}_n)$ by removing $K_{\frac{n}{d}-1}$ from it.\\
\noindent Let $u\in V(H)$  and $v, w\in V\left(K_{\frac{n}{d}-1}\right).$ By expressing $\Gamma(\mathbb{Z}_n)$ as the generalized join of $\mathscr{A}(d)$ we find that $u\sim v \sim w\sim u$  is a cycle (since $v$ and $w$ are the vertices of a complete graph). Therefore, $\Gamma(\mathbb{Z}_n)$ contains a cycle and hence it is not a tree. \\

\noindent \textbf{Case II:} Let $n=\prod\limits_{i=1}^kp_i$ and $p_i\neq 2$ for any $i$ and $k\geq 2.$ We will show that $\Gamma(\mathbb{Z}_n)$ is not a tree.

\noindent Let $p_i$ and $p_j$ be two distinct proper divisors of $n$ for $1\leq i,j\leq k.$ The sets $\mathscr{A}(p_i)$ and $\mathscr{A}(p_j)$ are given by:
\begin{align*}
    \mathscr{A}(p_i)=\{r_1p_i \hspace{2pt} | \hspace{2pt} r_1\in \mathbb{Z}_{\frac{n}{p_i}}, \hspace{2pt} gcd\left(r_1,\frac{n}{p_i}\right)=1\},\\
    \mathscr{A}(p_j)=\{r_2p_j \hspace{2pt} | \hspace{2pt} r_2\in \mathbb{Z}_{\frac{n}{p_j}}, \hspace{2pt} gcd\left(r_2,\frac{n}{p_j}\right)=1\}.
\end{align*}
Now the zero divisor graphs of $\mathscr{A}(p_i)$ and $\mathscr{A}(p_j)$ are $\overline{K}_{\frac{n}{p_i}-1}$ and $\overline{K}_{\frac{n}{p_j}-1},$ respectively. Since none of $p_i$ and $p_j$ is equal to two, hence both the null graphs contain more than two vertices each. Let $u_i, v_i\in V(\Gamma(\mathscr{A}(p_i)))$ and $u_j, v_j\in V(\Gamma(\mathscr{A}(p_j))).$ Let $d$ be a proper divisor of $n,$ then by expressing $\Gamma(\mathbb{Z}_n)$ as the generalized join of $\Gamma(\mathscr{A}(d)),$ we can get a cycle $u_i\sim u_j \sim v_i \sim v_j \sim u_i$ of length four.  Thus, $\Gamma(\mathbb{Z}_n)$ contains a cycle, and therefore, it is not a tree.\\

\noindent \textbf{Case III:} Let $n=2\prod\limits_{i=1}^kp_i,$ where $p_i\neq 2$ for any $i$ and $k\geq 2,$ then similar to case II, the zero divisor graph of $\mathbb{Z}_n$ contains a cycle and is not a tree.\\

\noindent \textbf{Case IV:} Let $n=2p,$ then we show that $\Gamma(\mathbb{Z}_n)$ is a tree.\\
\noindent The proper divisors of $2p$ are $2$ and $p.$\\   
\noindent Now define the following sets:
\begin{align*}
    \mathscr{A}(2)&=\{2r \hspace{2pt}| \hspace{2pt} r=1,2,\ldots, (p-1)\},\\
    \mathscr{A}(p)&=\{p\}.
\end{align*}
\noindent The zero divisor graph of $\mathscr{A}(2)$ and $\mathscr{A}(p)$ are $\overline{K}_{p-1}$ and $\overline{K}_1,$ respectively.\\
\noindent Therefore, ZDG of $\mathbb{Z}_{2p}$ can be expressed as the join of $\overline{K}_1$ and $\overline{K}_{p-1}.$ Figure \ref{2p} provides the visualization of $\Gamma(\mathbb{Z}_{2p}).$

\begin{center}
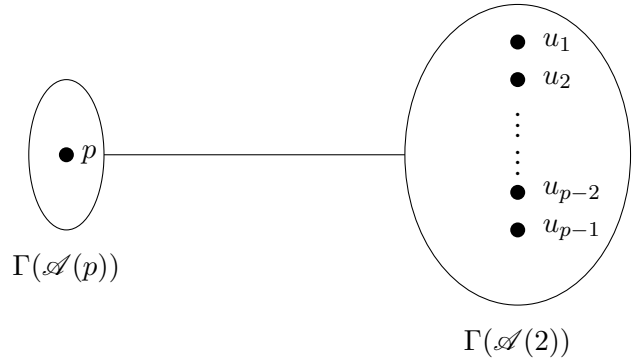

\begin{tikzpicture}
    
    \draw (0,0) ellipse (0.5cm and 1cm);
    \node at (0, -1.5) {$\Gamma(\mathscr{A}(p))$};

    \node at (0, 0) [circle, fill, inner sep=2pt] {};
    \node at (0.3, 0) {$p$};

    \draw (6,0) ellipse (1.5cm and 2.0cm);
    \node at (6, -2.5) {$\Gamma(\mathscr{A}(2))$};

    \node at (6, 1.5) [circle, fill, inner sep=2pt] {};
    \node[right] at (6.2, 1.5) {$u_1$};
    \node at (6, 1.0) [circle, fill, inner sep=2pt] {};
    \node[right] at (6.2, 1.0) {$u_2$};

    \node at (6, -0.5) [circle, fill, inner sep=2pt] {};
    \node[right] at (6.2, -0.5) {$u_{p-2}$};
    \node at (6, -1.0) [circle, fill, inner sep=2pt] {};
    \node[right] at (6.2, -1.0) {$u_{p-1}$};

    \node at (6, 0.5) {$\vdots$};
    \node at (6, 0) {$\vdots$};

    \draw (0.5,0) -- (4.5,0);
\end{tikzpicture}
\captionsetup{labelfont={bf},textfont={it}}
\captionof{figure}{The zero divisor graph of $\mathbb{Z}_{2p}.$}
\label{2p}
\end{center}
In the figure \ref{2p}, we can see that there a path $u_i \sim p \sim u_j,$ where $i\neq j.$ Clearly, there is no alternative path from $u_j$ to $u_i$ to complete a cycle. And hence there is no cycle in the graph. Therefore, $\Gamma(\mathbb{Z}_{2p})$ is a connected acyclic graph i.e. a tree. Moreover, it is seen that $\Gamma(\mathbb{Z}_{2p})$ is a star graph.\\

\noindent Analysing the all four cases it is proved that $\Gamma(\mathbb{Z}_n)$ is a tree if and only if $n=2p,$ where $p$ is a prime.
\end{proof}

\noindent  For a clearer understanding of tree structure derived from a ZDG, we now present an example.
\begin{example} Consider $p=7.$ Then the $\Gamma(\mathbb{Z}_{14})$ is a tree. The set of non-zero zero divisors of $\mathbb{Z}_{14}$ is $\{2,4,6,7,8,10,12\}.$ By analysing the adjacency condition for zero divisor graph we obtain the following graph.\\ 

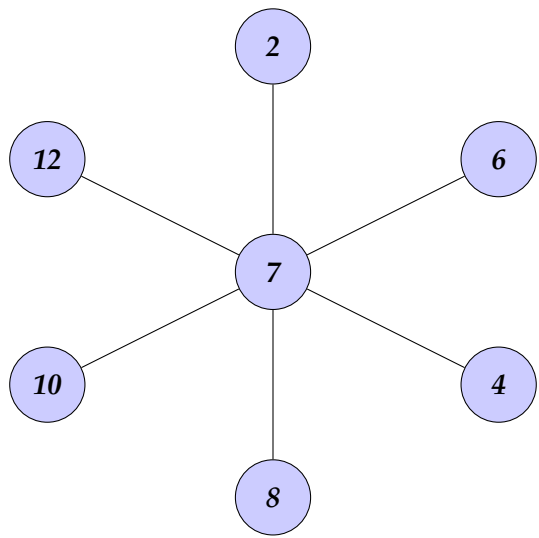
\centering
    \begin{tikzpicture}[scale=1.5, every node/.style={circle, draw, fill=blue!20, minimum size=10mm, inner sep=0pt, font=\bfseries}]
       
        \node (2) at (0, 2) {2};
        \node (4) at (2, -1) {4};
        \node (6) at (2, 1) {6};
        \node (7) at (0, 0) {7};
        \node (8) at (0, -2) {8};
        \node (10) at (-2, -1) {10};
        \node (12) at (-2, 1) {12};

        \draw (2) -- (7);
        \draw (4) -- (7);
        \draw (6) -- (7);
        \draw (8) -- (7);
        \draw (10) -- (7);
        \draw (12) -- (7);
    \end{tikzpicture}
    \captionsetup{labelfont={bf},textfont={it}}
\captionof{figure}{Zero Divisor Graph of     \(\mathbb{Z}_{14}\)}

\end{example}

\noindent 
Example \ref{Eg3} in the appendix illustrates that if  $n$ is not in the form of $2p,$ where $p$ is a prime, then the respective ZDG of $\mathbb{Z}_n$ is not a tree.\\

\noindent The following corollaries, which immediately follow from above theorems \ref{least}, \ref{irreducible} and \ref{tree}, provide deeper understandings into the eccentricity matrix of a zero divisor graph. They bridge our study with existing literature and offer a foundation for further exploration of zero divisor graphs.

\begin{corollary}
The least eigenvalue of the zero divisor graph of the eccentricity matrix of $\mathbb{Z}_{2p}$ is $-2,$ where $p$ is a prime.
\end{corollary}

\begin{corollary}
If $p$ is a prime, then the eccentricity matrix of the zero divisor graph of $\mathbb{Z}_{2p}$ is irreducible.
\end{corollary}

\noindent We have seen that out of all possible values of $n,$ only for $n=2p,$ the ZDG of $\mathbb{Z}_n$ is a tree. In the following section we turn our attention to another fundamental question in the study of ZDGs: determining the conditions under which the eccentricity eigenvalues of  $\Gamma(\mathbb{Z}_n)$ are integers.

\section{Integrality of eccentricity eigenvalues of $\Gamma(\mathbb{Z}_{p^t}),$ $\Gamma_E(\mathbb{Z}_{p^t})$ and $\Gamma(\mathbb{Z}_p \times \mathbb{Z}_p)$}
\label{Integrality}

After investigating the eccentricity spectra of zero-divisor graphs, a natural question arises: for which values of 
$n$  does  the $\Gamma(\mathbb{Z}_{n})$ have integral eccentricity eigenvalues? 

\noindent In this section, we try to answer this question for $n=p^t,$ where $p$ is a prime and $t$ is an integer. We find the condition on $t$ for which the eccentricity eigenvalues of $\Gamma(\mathbb{Z}_{p^t})$ and $\Gamma_E(\mathbb{Z}_{p^t})$  are integers.

\subsection{Integrality of eccentricity eigenvalues of $\Gamma(\mathbb{Z}_{p^t})$}
Here, in this subsection we discuss the integrality of eccentricity eigenvalues of $\Gamma(\mathbb{Z}_{p^t}).$

\begin{theorem}
    Let $p$ be a prime, then the eccentricity eigenvalues of $\Gamma(\mathbb{Z}_{p^t})$ are integers if and only if $t=2.$
\end{theorem} 

\begin{proof} 
For $t=1,$ $\mathbb{Z}_p$ is a field and hence $\Gamma(\mathbb{Z}_{p})$ is a null graph. Therefore, the case is trivial.\\

\noindent For $t=2,$ the ZDG of $\mathbb{Z}_{p^2}$ is a complete graph of order $p-1.$ Hence the eccentricity matrix and the adjacency matrix of $\Gamma(\mathbb{Z}_{p^2})$ coincide. Therefore, the eccentricity spectrum of $\Gamma(\mathbb{Z}_{p^2})$ is given by  
\begin{align*}
      \begin{Bmatrix}
    p-2 & -1\\
    1 & p-2
\end{Bmatrix}.  
    \end{align*}
So, $\Gamma(\mathbb{Z}_{p^2})$ has integral eccentricity spectrum.\\

\noindent For $t\geq 3,$ the structure of $\Gamma(\mathbb{Z}_{p^t})$ becomes more complicated. We analyze the cases separately for even and odd $t.$\\

\noindent \textbf{Case 1:} Let $G=\Gamma(\mathbb{Z}_{p^{2k}}).$ If $t=2k$ $(k>1)$, the proper divisors of $p^{2k}$ are $p, p^2, \ldots, p^{2k-1}.$ The zero-divisor graph $\Gamma(\mathbb{Z}_{p^{2k}})$ can be recursively constructed as:

$$G=\Gamma(\mathbb{Z}_{p^{2k}})=\Upsilon_{p^{2k}}[\Gamma(A_p), \Gamma(A_{p^2})\cdots, \Gamma(A_{p^{2k-1}})]$$.

 \noindent where each $\Gamma(A_{p^j})$(From \cite[Corollary 2.5]{chat}) is given by:

\begin{align*}
 \Gamma(A_{p^j})=\begin{cases}
       \overline K_{\phi(p^{2k-j})} \quad for \hspace{3pt} 1\leq j\leq k-1\\
       K_{\phi(p^{2k-j})} \quad for \hspace{3pt} k\leq j\leq 2k-1.     
    \end{cases}
\end{align*}
From the definition it follows that in the graph $\Upsilon_{p^{2k}},$ two vertices $p^i$ and $p^j$ are adjacent for all $1 \leq i \leq 2m-1$ and $j \geq 2m-i$ with $i \neq j.$ We define $\Upsilon_{p^{2k}}$ recursively to illustrate its construction. Now, $\Upsilon_{p^{2k}}$ can be defined recursively as follows:
\begin{align*}
    &G_1=\{p^k\}\\
    &G_2=\{p^{k+1}\} \lor [\{p^{k-1}\} \cup G_1]\\
    &\vdots\\
    &G_k=\{p^{2k-1}\} \lor [\{p\} \cup G_{k-1}], 
\end{align*}
\noindent where $\{x\}$ denotes the graph with one vertex $x.$
Therefore, $G_k \cong \Upsilon_{p^{2k}}.$ Now we define $H_1, H_2, \ldots, H_k$ as follows:
\begin{align*}
    &H_1=K_{\phi(p^m)}\\
    &H_2=K_{\phi(p^{m-1})} \lor [\overline{K_{\phi(p^{m+1})}} \cup H_1]\\
    &\vdots\\
    &H_k=K_{\phi(p)} \lor [\overline{K_{\phi(p^{2m-1})}} \cup H_{k-1}]. 
\end{align*}
Therefore, the graph $\Gamma(\mathbb{Z}_{p^{2k}}) \cong H_k.$\\

\noindent To prove the eccentricity spectrum is non-integral for $t=2k$, we use induction on $k$.\\

\noindent  For $k=2,$ consider $\Gamma(\mathbb{Z}_{p^4})$. Its eccentricity matrix is:

\begin{align*}
    \varepsilon(G)=\begin{bmatrix}
        O_{\phi(p^2) \times \phi(p^2)} & 2J_{\phi(p^2) \times \phi(p^{3})} & J_{\phi(p^2) \times \phi(p)}\\
        2J_{\phi(p^{3}) \times \phi(p^2)} & 2(J-I)_{\phi(p^{3}) \times \phi(p^{3})} & J_{\phi(p^{3}) \times \phi(p)}\\
        J_{\phi(p) \times \phi(p^2)} & J_{\phi(p) \times \phi(p^{3})} & (J-I)_{\phi(p) \times \phi(p)}
    \end{bmatrix}.
\end{align*}
Since the row sum of each block is constant,  so  the corresponding quotient matrix is:
\begin{align*}
    Q(G)=\begin{bmatrix}
       0 & 2(p^3-p^2) & p+1\\
       p^2-p & 2(p^3-p^2-1) & p+1\\
       p^2-p & p^3-p^2 & p-2
    \end{bmatrix}.
\end{align*}

\noindent The eigenvalues of $Q(G)$ are the roots of the characteristic equation $$2p - 7 p^3 + 11 p^4 - 7p^5 + p^6 + (4 - p + 5 p^2 - 9 p^3 + 5 p^4 - 2 p^5) x + (4 - p + 
    2 p^2 - 2 p^3) x^2 + x^3=0.$$
\noindent This equation has one irrational root and two complex roots. And hence $Q(G)$ does not have any integral eccentricity eigenvalue. By \cite{brouwer}, we get that $\sigma(Q(G))\subset \sigma(\epsilon(G)).$ Hence the spectrum of $\Gamma(\mathbb{Z}_{p^4})$ is not integral.\\

\noindent Assume that for $k\leq r,$ the matrix $\epsilon(\Gamma(\mathbb{Z}_{p^{2r}}))$ has non integral spectrum.\\

\noindent We need to show that $\epsilon(\Gamma(\mathbb{Z}_{p^{2r}}))$ has non integral spectrum for $k=r+1.$ \\

\noindent The eccentricity matrix of $\Gamma(\mathbb{Z}_{p^{2r+2}})$ is given by 
\begin{align*}
    \varepsilon(G)=\begin{bmatrix}
        O_{\phi(p^2) \times \phi(p^2)} & 2J_{\phi(p^2) \times \phi(p^{3})} & \ldots& J_{\phi(p^2) \times \phi(p)}\\
        2J_{\phi(p^{3}) \times \phi(p^2)} & 2(J-I)_{\phi(p^{3}) \times \phi(p^{3})} & \ldots & J_{\phi(p^{3}) \times \phi(p)}\\
        \vdots & \vdots & \ddots & \vdots\\
        J_{\phi(p) \times \phi(p^2)} & J_{\phi(p) \times \phi(p^{3})} & \ldots & (J-I)_{\phi(p) \times \phi(p)}
    \end{bmatrix}.
\end{align*}
The quotient matrix of the above matrix is given by:
\begin{align*}
    Q(G)=\begin{bmatrix}
       0 & 2(p^3-p^2) & \ldots & p+1\\
       p^2-p & 2(p^3-p^2-1)& \ldots & p+1\\
       \vdots & \vdots & \ddots & \vdots\\
       p^2-p & p^3-p^2& \ldots & p-2
    \end{bmatrix}.
\end{align*}
Similar to previous quotient matrix, we can see that $Q(G)$ has non-integral eigenvalues. By \cite{brouwer}, $\epsilon(G)$ has non-integral eigenvalues and hence $\Gamma(\mathbb{Z}_{p^{2r+2}})$ has non-integral spectrum.\\

\noindent \textbf{Case 2:} If $t=2k+1(k\geq 0),$ the proper divisors of $p^{2k+1}$ are $p, p^2, \ldots, p^{2k}.$ Similar to even case, $\Upsilon_{p^{2k+1}}$ can be expressed recursively as the join and union of $\{p^m\},$ for some suitable $m$. By constructing the eccentricity matrix and examining its structure, one can show in this case, through similar arguments, that it also has no integral eccentricity eigenvalues.\\

\noindent From all cases, we conclude that the eccentricity eigenvalues of $\Gamma(\mathbb{Z}_{p^t})$ are integers if and only if $t = 2$. 
\end{proof}

\begin{note}
    In particular, for $t=3$ and $4,$ theorem \ref{Theorem 3.2} and theorem \ref{Theorem 3.4} show that the spectrum of $\Gamma(\mathbb{Z}_{p^3})$ and $\Gamma(\mathbb{Z}_{p^4})$ contain non-integral eigenvalues.
\end{note}

\subsection{Integrality of eccentricity eigenvalues of $\Gamma_E(\mathbb{Z}_{p^t})$}
In this subsection, we investigate the condition on $t$ for which the eccentricity eigenvalues of $\Gamma_E(\mathbb{Z}_{p^t})$ are integers.
\begin{theorem}
    Let $p$ be any prime, then the eccentricity spectrum of $\Gamma_E(\mathbb{Z}_{p^t})$ is integral if $t\geq 2.$
\end{theorem}

\begin{proof}
The vertex set of the graph $\Gamma_E(\mathbb{Z}_{p^t})$ is $Z(\mathbb{Z}_{p^t})^*$, which consists of the non-zero zero divisors of $\mathbb{Z}_{p^t}$. This set can be expressed as 
   
    \begin{align*}
        Z^*(\mathbb{Z}_{p^t})= 
        \{kp \hspace{2pt}|  \hspace{2pt} 1\leq k \leq p^{t-1}-1\}.
    \end{align*}
Consider two distinct elements $k_1p$ and $k_2p$ in $Z(\mathbb{Z}_{p^t})^*$. For some positive integers $m$ and $n$, if $m + n \geq t$, then
    \[
        (k_1p)^m (k_2p)^n = k_1^m k_2^n p^{m+n}.
    \]
    Since $m + n \geq t$, it follows that $p^{m+n}$ is divisible by $p^t$, and hence $(k_1p)^m (k_2p)^n = 0$ in $\mathbb{Z}_{p^t}$.
 Thus, any two distinct vertices of $\Gamma_E(\mathbb{Z}_{p^t})$ are adjacent, which implies that $\Gamma_E(\mathbb{Z}_{p^t})$ is a complete graph on $\lvert Z(\mathbb{Z}_{p^t})^*\rvert$ vertices.

 \noindent The eccentricity matrix of the complete graph $\Gamma_E(\mathbb{Z}_{p^t})$ is given by
    $$
    \varepsilon(\Gamma_E(\mathbb{Z}_{p^t})) = \begin{bmatrix}
        0 & 1 & \ldots & 1 \\
        1 & 0 & \ldots & 1 \\
        \vdots & \vdots & \ddots & \vdots \\
        1 & 1 & \ldots & 0
    \end{bmatrix}.
    $$
Hence, the eccentricity spectrum  is
\begin{align*}
    \sigma(\Gamma_E(\mathbb{Z}_{p^t}))=\begin{Bmatrix}
    -1 & & \lvert Z^*(\mathbb{Z}_{p^t})\rvert-1\\
    \lvert Z^*(\mathbb{Z}_{p^t})\rvert-1 & & 1
\end{Bmatrix}.
\end{align*}
Therefore, the eccentricity eigenvalues of $\Gamma_E(\mathbb{Z}_{p^t})$ are integers for $t \geq 2$.
\end{proof}

\subsection{Integrality of eccentricity eigenvalues of $\Gamma(\mathbb{Z}_{p} \times \mathbb{Z}_p)$}
In this subsection, we show that the eccentricity spectra of $\Gamma(\mathbb{Z}_{p} \times \mathbb{Z}_p)$ consists solely of integers.
\begin{theorem}
    If $p$ is a prime, then the eccentricity spectrum of $G=\Gamma(\mathbb{Z}_{p} \times \mathbb{Z}_p)$ is given by 
    $$\sigma(\varepsilon(G))=\begin{Bmatrix}
     -2 & 2p-6\\
     2(p-1) & 2   
    \end{Bmatrix}$$
\end{theorem}

\begin{proof}
    Let $G=\Gamma(\mathbb{Z}_{p} \times \mathbb{Z}_p)$. The graph $G$ can be decomposed as $$G=P_2[\overline K_{(p-1)}, \overline K_{(p-1)}].$$
    The eccentricity matrix of the graph $G$ is 
    $$\varepsilon(G)=\begin{bmatrix}
        2(J-I)_{(p-1)\times (p-1)} & O_{(p-1)\times (p-1)}\\
        O_{(p-1)\times (p-1)} & 2(J-I)_{(p-1)\times (p-1)}
    \end{bmatrix}=\begin{bmatrix}
        M_{11} & O\\
        O & M_{11}
    \end{bmatrix},$$
    where $M_{11}=2(J-I)_{(p-1)\times (p-1)}$.\\
     Using the lemma \ref{Schur}, the characteristic polynomial of $\varepsilon(G)$ is given by
    \begin{align*}
        P_{\varepsilon(G)}(x)=\det(xI-M_{11})\cdot \det(xI-M_{11}).
     \end{align*}
  Further, applying lemmas  \ref{constant coronel} and \ref{det}, we get
    \begin{align*}
        \det(xI-M_{11}) &= \det(xI-2J+2I),\\
        &= (1-2\Gamma_{2J}(x))\cdot \det(xI-2J),\\
        &= (x+2)^{(p-1)}[x-2(p-1-2)],\\
        &= (x+2)^{(p-1)}[x-(2p-6)].
    \end{align*}

\noindent Thus,  the characteristic polynomial is:
$$P_{\varepsilon (G)}(x)=(x+2)^{2(p-1)}[x-(2p-6)]^2.$$
Therefore, the eccentricity spectrum of $G=\Gamma(\mathbb{Z}_{p} \times \mathbb{Z}_p)$ is: 
$$\sigma(\varepsilon(G))=\begin{Bmatrix}
     -2 & 2p-6\\
     2(p-1) & 2   
    \end{Bmatrix}.$$
This shows that the eccentricity eigenvalues of $\Gamma(\mathbb{Z}_p\times \mathbb{Z}_p)$ are indeed integers.
\end{proof}

\vspace{.59cm}

\noindent As an immediate consequence,  the following corollary formally states the spectral radius of $\Gamma(\mathbb{Z}_p \times \mathbb{Z}_p).$

\begin{corollary}
    For a prime $p$, the eccentricity spectral radius of $\Gamma(\mathbb{Z}_p \times \mathbb{Z}_p)$ is $2p-6.$
\end{corollary}

\noindent The eccentricity energy gap between the complement and zero divisor graph of $\mathbb{Z}_n$ is examined in the following section along with the spectral and structural differences between them.\\

\section{The absolute eccentricity energy gap between $\Gamma(\mathbb{Z}_n)$ and its complement}
\label{Gap}

This section primarily focuses on the eccentricity energy gap between ZDG of $\mathbb{Z}_{p_1p_2}$ and its complement, where $p_1$ and $p_2$ are distinct primes. We give an upper bound on this energy gap by examining the spectra of these graphs.In chemical graph theory, this gap is very helpful for evaluating molecule reactivity, stability, and network design optimization \cite{machineLearn}. Notable applications and examples can be found in \cite{niki}.\\ 
 
\noindent To understand the eccentricity energy of the complement graph $\overline{G}$,  we need to understand the following lemma about graph spectra:

\begin{lemma}\cite{unionG}
    The spectrum of the union of two different graphs is the union of the spectrum of the two graphs.
\end{lemma}

\begin{theorem}
    For two distinct primes $p_1$ and $p_2$, the eccentricity energy of $\overline G$, where $G=\Gamma(\mathbb{Z}_{p_1p_2})$, is given by 
    $$\mathscr{E}(\overline G)=2(p_1+p_2-4).$$
\end{theorem}

\begin{proof}
    The complement of $\Gamma(\mathbb{Z}_{p_1p_2})$ can be realized as 
    $$\overline G=\overline{\Gamma(\mathbb{Z}_{p_1p_2})} = K_{p_1-1} \bigcup K_{p_2-1}.$$
    Therefore, $$\varepsilon(\overline G)=\begin{bmatrix}
        (J-I)_{p_1-1\times p_1-1} & O\\
        O & (J-I)_{p_2-1\times p_2-1}
    \end{bmatrix}.$$
The spectrum of $ \varepsilon(\overline{G}) $ can be derived from the well-known spectra of the matrices $J - I$ (using remark \ref{constant coronel} and lemma \ref{det}). Consequently, it is given by    

$$\begin{Bmatrix}
        p_1-2 & p_2-2 & -1\\
        1 & 1 & p_1+p_2-4
\end{Bmatrix}.$$

    Thus, the eccentricity energy of $\overline{G}$, which is the sum of the absolute values of these eigenvalues, is
$$\mathscr{E}(\overline{G}) = 2(p_1 + p_2 - 4).$$
    
\end{proof}

\begin{theorem}
    For a prime $p$, the eccentricity energy of $\overline G =\overline{\Gamma(\mathbb{Z}_{p^3})}$ is given by $$\mathscr{E}(\overline G)=2p(p-1)-2.$$
\end{theorem}
\begin{proof}
 The complement of \( \Gamma(\mathbb{Z}_{p^3}) \) can be expressed as:
    $$\overline{G} = \overline{\Gamma(\mathbb{Z}_{p^3})} = K_{p(p-1)} \cup \overline{K}_{p-1}.$$
    Thus, the eccentricity matrix of $\overline{G}$ is
   
     $$\varepsilon(\overline G)=\begin{bmatrix}
        (J-I)_{p(p-1)\times p(p-1)} & O_{p(p-1) \times p-1}\\
        O_{p-1 \times p(p-1)} & O_{p-1 \times p-1}
    \end{bmatrix}.$$
    The spectrum of this matrix follows from the spectra of $ J - I$ matrices (using remark \ref{constant coronel} and lemma \ref{det}) and is given by   $$\begin{Bmatrix}
        p(p-1)-1 & -1\\
        1 & p(p-1)-1
    \end{Bmatrix}.$$
    Hence, the eccentricity energy of $\overline G$ is $$\mathscr{E}(\overline G)=2p(p-1)-2.$$
\end{proof}

\begin{theorem}
\label{4.3}
    For two distinct primes $p_1$ and $p_2$, let $G=\Gamma(\mathbb{Z}_{p_1p_2})$. The absolute eccentricity energy gap between $G$ and $\overline G$ is given by 
    $$\lvert \mathscr{E}(G) - \mathscr{E}(\overline G) \rvert \leq 3(p_1+p_2-2)^2.$$
\end{theorem}

\begin{proof}
 Consider the matrices 
    \[
    \varepsilon(G) = A = (a_{ij}) \quad \text{and} \quad \varepsilon(\overline{G}) = B = (b_{ij}).
    \]
    From the eccentricity matrices $A$ and $B,$ it is straight forward that the maximum absolute values of the entries in \(A\) and \(B\) are \(2\) and \(1\), respectively. i.e.
    \begin{align*}
    \max |a_{ij}| = 2 \hspace{10pt} \rm{and} \hspace{10pt} \max |b_{ij}| = 1.
    \end{align*}

\noindent We have,

$$ \mathscr{E}(A)=\sum\limits_{i=1}^{p_1+p_2-2}\lambda_i \hspace{20pt} \rm{and} \hspace{20pt} \mathscr{E}(B)=\sum\limits_{i=1}^{p_1+p_2-2}\mu_i,$$
\noindent  where $\lambda_i$ and $\mu_i$ are the eigenvalues of $A$ and $B,$ respectively for $1\leq i \leq p_1+p_2-2.$\\

\noindent If $\lambda$ be an eigenvalue of $A$, we can use $\lambda$ to demonstrate the calculation of the maximum possible eigenvalue in this context. Let $x$ be  a unit vector. So, we have 
    \begin{align*}
        (x, Ax)=(x, A^*x)=\lambda(x,x),
    \end{align*}
    where $A^*$ is the conjugate transpose of $A$.\\\\
    By applying the Cauchy- Schwarz inequality,
    \begin{align*}
        \lvert \lambda(x,x) \rvert = \lvert \lambda \rvert \lvert (x,x) \rvert = \lvert (x, Ax) \rvert &\leq \sum \limits_{i=1}^{p_1+p_2-2} \sum \limits_{j=1}^{p_1+p_2-2} \lvert a_{ij}\rvert \lvert x_i\rvert \lvert x_j\rvert,
    \end{align*}
    Since, $a$ is positive and the maximum of $a_{ij},$ therefore,
        \begin{align*}
        \sum \limits_{i=1}^{p_1+p_2-2} \sum \limits_{j=1}^{p_1+p_2-2} \lvert a_{ij}\rvert \lvert x_i\rvert \lvert x_j\rvert &\leq a(\sum \limits_{i=1}^{p_1+p_2-2} \sum \limits_{j=1}^{p_1+p_2-2} \lvert x_i\rvert \lvert x_j\rvert),\\
        &\leq \frac{a}{2}\sum \limits_{i=1}^{p_1+p_2-2} \sum \limits_{j=1}^{p_1+p_2-2} (\lvert x_i\rvert^2 +\lvert x_j\rvert^2) \hspace{8pt}\left[\rm{since,} \hspace{4pt} (a+b)^2\geq 0\right],\\
        &= \frac{a}{2}\sum \limits_{i=1}^{p_1+p_2-2} \sum \limits_{j=1}^{p_1+p_2-2} \lvert x_i\rvert^2 + \frac{a}{2}\sum \limits_{i=1}^{p_1+p_2-2} \sum \limits_{j=1}^{p_1+p_2-2}\lvert x_j\rvert^2,\\
        &= \frac{(p_1+p_2-2)a}{2}+ \frac{(p_1+p_2-2)a}{2},\\
        &=(p_1+p_2-2)a,\\
        &=2(p_1+p_2-2) \hspace{8pt}\left[\rm{since,} \hspace{4pt} a=2\right].
    \end{align*}
    Thus, $$\lvert \lambda\rvert \leq 2(p_1+p_2-2).$$
Therefore, the total eccentricity energy is:
     $$\mathscr{E}(A)=\sum\limits_{i=1}^{p_1+p_2-2}\lambda_i \leq 2(p_1+p_2-2)^2.$$
Similarly, for \(B = \varepsilon(\overline{G})\), where the maximum value of \(\lvert b_{ij} \rvert\) is \(1\):
     
     $$\mathscr{E}(B) \leq (p_1+p_2-2)^2.$$
    Hence,  the absolute energy gap of $\varepsilon(G)$ and $\varepsilon(\overline G)$ is :
    \begin{align*}
        \lvert \mathscr{E}(A) - \mathscr{E}(B) \rvert &\leq \lvert \mathscr{E}(A) \rvert + \lvert \mathscr{E}(B) \rvert,\\
        &\leq 3(p_1+p_2-2)^2,\\
\end{align*}
Therefore, 
$$\lvert \mathscr{E}(G) - \mathscr{E}(\overline G) \rvert \leq 3(p_1+p_2-2)^2.$$    
\end{proof}

\begin{theorem}
    For a prime $p$, let $G=\Gamma(\mathbb{Z}_{p^3})$. The absolute eccentricity energy gap between $G$ and $\overline G$ is given by 
    $$\lvert \mathscr{E}(G) - \mathscr{E}(\overline G) \rvert \leq 3(p^2-1)^2.$$
\end{theorem}

\begin{proof}
    The proof is similar to theorem \ref{4.3}.
\end{proof}

\appendix
\section*{Appendices}

\section{Detailed Calculations}

\subsection{Calculation of Theorem 3.3}
\label{appendix:th A.3}
The detail calculation of theorem \ref{Theorem 3.3} is shown here.\\

\noindent The characteristic polynomial of the matrix $\varepsilon(G)$ is \begin{align*}
        P_{\varepsilon(G)}(x)&=\det(\varepsilon(G)-xI),\\
        &= \begin{vmatrix}
            2(J-I)-xI & O & 2J\\
            O & -xI & O\\
            2J & O & -xI
        \end{vmatrix},\\
        &= \begin{vmatrix}
        M_{11} & O & M_{13}\\
        O & M_{22} & O\\
        M_{31} & O & M_{33}
    \end{vmatrix},
        \end{align*}
where, 
\begin{align*}
     M_{11}&=\left(2(J-I)-xI\right)_{p^2(p-1)\times p^2(p-1)}, \hspace{5mm} &M_{13}&=2J_{p^2(p-1)\times p(p-1)},\\
     M_{22}&=-xI_{(p-1)\times (p-1)}, \hspace{5mm} &M_{31}&=2J_{p(p-1)\times p^2(p-1)},\\
     M_{33}&=-xI_{p(p-1)\times p(p-1)}.
\end{align*}
    Using Lemma \ref{Schur} we get 
    \begin{align}
        P_{\varepsilon(G)}(x)&=\det\left(\left[M_{11}-M_{13}M_{33}^{-1}M_{31}\right]\right)\cdot \det(M_{22})\cdot \det(M_{33}).
        \label{poly3}
    \end{align}
    Now, 
    \begin{align*}
        \det(M_{22})=x^{p-1}, \hspace{7mm} \det(M_{33})=x^{p(p-1)}.
    \end{align*}
    Let \begin{align*}
        M&=M_{11}-M_{13}M_{33}^{-1}M_{31}\\
        &=\begin{bmatrix}
            -x + \frac{4p(p-1)}{x} & 2+ \frac{4p(p-1)}{x} & \cdots & 2+\frac{4p(p-1)}{x}\\
        2+ \frac{4p(p-1)}{x} & -x +\frac{4p(p-1)}{x} & \cdots & 2+\frac{4p(p-1)}{x}\\
        \vdots & \vdots & \ddots & \vdots\\
        2+\frac{4p(p-1)}{x} & 2+\frac{4p(p-1)}{x} & \cdots & -x + \frac{4p(p-1)}{x}
        \end{bmatrix}.
    \end{align*}
    By using Lemma \ref{det} we get 
    $$\det(M)=\frac{(x +2)^{p^2(p-1)}\left[-x ^2+x(2p^3-2p-2)+4p^3(p-1)^2\right]}{x}.$$
    From the equation \ref{poly3} we get
    $$P_{\varepsilon(G)}(x)= x^{p^2-1}(x +2)^{p^2(p-1)}\left[-x ^2+x(2p^3-2p-2)+4p^3(p-1)^2\right].$$
Therefore the eccentricity spectrum of the graph $G=\Gamma(\mathbb{Z}_{p^4})$ is
\begin{align*}
\sigma(\varepsilon(\Gamma(\mathbb{Z}_{p^4})))=&\begin{Bmatrix}
   -2 & 0 & 1-p-p^3-\Lambda & 1-p-p^3+\Lambda\\ p^2(p-1) & p^2-1 & 1 & 1
   \end{Bmatrix},
\end{align*}
where, $\Lambda=\sqrt{2+2p+p^2+2p^3-10p^4+4p^5+p^6}.$

\subsection{Calculation to the theorem 3.4}
\label{Appendix A.4}
The detail computation is shown here for the theorem \ref{Theorem 3.4}.\\

\noindent Using the Lemma \ref{Schur} we get
\begin{align*}
    P_{\varepsilon(G)}(x)= &\det(-xI)\cdot \det(M_{22}-xI)\cdot \det(M_{33}-xI)\cdot\\ &\det\left(-xI+M_{41}(-xI)^{-1}M_{14}+M_{42}(M_{22}-xI)^{-1}M_{24}\right).
\end{align*}
Now, $$\det(-xI)=x^{p_1^2-1}.$$
Using Lemma \ref{constant coronel} and Lemma \ref{det} we get
\begin{align*}
    \det(M_{22}-xI)=(x+2)^{p_2-1}\left[x-2(p_2-1-2)\right].
\end{align*}
Similarly, \begin{align*}
    \det(M_{33}-xI)=(x+2)^{(p_1-1)(p_2-1)}\left[x-2\{(p_1-1)(p_2-1)-2\}\right].
\end{align*}
And,
\begin{align*}
    &\det(-xI+M_{41}(-xI)^{-1}M_{14}+M_{42}(M_{22}-xI)^{-1}M_{24})\\
    &=(-x)^{p_1-2}\left[\frac{-9x^2(p_2-1)^2-4x^2(p_2-1)+18x(p_2-1)^2-44x(p_2-1)+36(p_2-1)^2-36(p_2-1)}{x^3-2x^2(p_2-3)-4(p_2-2)}\right]\\
    &\hspace{10pt} -(-x)^{p_1-2}(p_2-1)x \tag{A}
    \label{A}
\end{align*}
Hence,
\begin{align*}
    &P_{\varepsilon(G)}(x)= x^{p_1^2+2p_1-4}(x+2)^{p_1(p_2-1)}\left[x-2(p_2-1-2)\right]\left[x-2\{(p_1-1)(p_2-1)-2\}\right]\\
    &\left[\frac{-9x^2(p_2-1)^2-4x^2(p_2-1)+18x(p_2-1)^2-44x(p_2-1)+36(p_2-1)^2-36(p_2-1)}{x^3-2x^2(p_2-3)-4(p_2-2)}-(p_2-1)x\right]
\end{align*}
Therefore the eccentricity spectrum of the graph $G$ is 
\begin{align*}
\sigma(\varepsilon(G))= \begin{Bmatrix}
     0 & -2 & 2p_2-6 & 2(p_1-1)(p_2-1)-4\\
     p_1^2+2p_1-4 & p_1(p_2-1) & 1 & 1
    \end{Bmatrix} \bigcup \hspace{5pt} \Theta_{P_{\varepsilon(G)}(x)},
\end{align*}
where $\Theta_{P_{\varepsilon(G)}(x)}$ is the set of distinct roots of the polynomial (\ref{A})  with their multiplicities.

\section{Some examples to understand the decomposition of ZDG and their eccentricity spectra}
\label{appendix:B}
Here we provide some examples to understand the decomposition of ZDG and their spectra in a better way
\begin{example}
\label{Eg1}
    If $p=2,$ then the eccentricity spectra of $\Gamma(\mathbb{Z}_{p^3})$ is given by 
    \begin{align*}
    \sigma(\varepsilon(\Gamma(\mathbb{Z}_{8})))=\begin{Bmatrix}
        -2 & 1+\sqrt{3} & 1-\sqrt{3}\\
        1 & 1 & 1
    \end{Bmatrix}.
    \end{align*}

\noindent We know that the set of zero divisors of the ring $\mathbb{Z}_8$ is $\{2,4,6\}.$ The zero divisor graph of this ring is shown in figure $2.$

\begin{center}
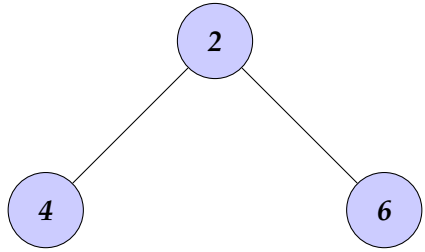

\begin{tikzpicture}[scale=1.5, every node/.style={circle, draw, fill=blue!20, minimum size=10mm, inner sep=0pt, font=\bfseries}]

\node[circle, draw] (A) at (0,0) {2};
\node[circle, draw] (B) at (-1.5,-1.5) {4};
\node[circle, draw] (C) at (1.5,-1.5   ) {6};

\draw (A) -- (B);
\draw (A) -- (C);

\end{tikzpicture}
\captionsetup{labelfont={bf},textfont={it}}
\captionof{figure}{The zero divisor graph of $\mathbb{Z}_8.$}
\end{center}
The eccentricity matrix of this graph is
\begin{align*}
    \varepsilon(\Gamma(\mathbb{Z}_8))=\begin{bmatrix}
        0 & 1 & 2\\
        1 & 0 & 1\\
        2 & 1 & 0
    \end{bmatrix}.
\end{align*}
The characteristic polynomial of this matrix is 
\begin{align*}
    P_{\varepsilon(\Gamma(\mathbb{Z}_8))}(x)=x^3-4x-6.
\end{align*}
Hence, the spectrum of $\Gamma(\mathbb{Z}_8)$ is 
 \begin{align*}
    \sigma\left(\varepsilon(\Gamma(\mathbb{Z}_{8}))\right)=\begin{Bmatrix}
        -2 & 1+\sqrt{3} & 1-\sqrt{3}\\
        1 & 1 & 1
    \end{Bmatrix}.
    \end{align*}
\end{example}
\vspace{8pt}

\begin{example}
\label{Eg2}
    Consider the ring $Z_{8}.$ We show that the eccentricity spectrum of the extended zero divisor graph, $\Gamma_E(\mathbb{Z}_8)$ is 
    $$\sigma_\varepsilon(\Gamma_E(\mathbb{Z}_8))=\begin{Bmatrix}
        -1 & 2\\
        2 & 1
    \end{Bmatrix}.$$
    The zero divisors of $\mathbb{Z}_8$ are $\{2,4,6\}.$ The extended zero divisor graph of $\mathbb{Z}_8$ can be decomposed as 
    $$\Gamma_E(\mathbb{Z}_8)=P_2\left[\overline{K}_1, K_2\right],$$
    where $P_2,$ $\overline{K}_1$ and $K_2$ are the path on $2$ vertices, null graph on $1$ vertex and complete graph on $2$ vertices, respectively.\\

\begin{center}
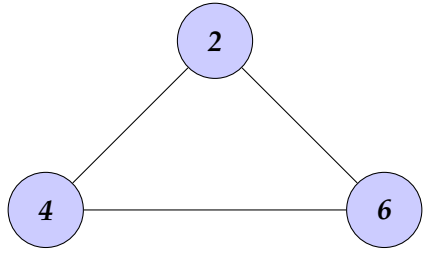

\begin{tikzpicture}[scale=1.5, every node/.style={circle, draw, fill=blue!20, minimum size=10mm, inner sep=0pt, font=\bfseries}]

\node[circle, draw] (A) at (0,0) {2};
\node[circle, draw] (B) at (-1.5,-1.5) {4};
\node[circle, draw] (C) at (1.5,-1.5) {6};

\draw (A) -- (B);
\draw (A) -- (C);
\draw (B) -- (C);

\end{tikzpicture}
\captionsetup{labelfont={bf},textfont={it}}
\captionof{figure}{The extended zero divisor graph of $\mathbb{Z}_8.$}
\end{center}
Therefore, the eccentricity matrix of $\Gamma_E(\mathbb{Z}_8)$ is
$$\varepsilon(\Gamma_E(\mathbb{Z}_8))=\begin{bmatrix}
    0 & 1 & 1\\
    1 & 0 & 1\\
    1 & 1 & 0
\end{bmatrix}.$$
The characteristic polynomial of this matrix is $x^3-3x-2.$
Hence the spectrum is 
$$\sigma_\varepsilon(\Gamma_E(\mathbb{Z}_8))=\begin{Bmatrix}
        -1 & 2\\
        2 & 1
    \end{Bmatrix}.$$   
\end{example}
\vspace{8pt}

\begin{example}
\label{Eg3}
Let $n=3\times 5=15.$ Then we demonstrate that $\Gamma(\mathbb{Z}_{15})$ is not a tree. The proper divisors of $15$ are $3$ and $5.$ Therefore,
\begin{align*}
    \mathscr{A}(3)&=\{3,6,9,12\},\\
    \mathscr{A}(5)&=\{5,10\}.
\end{align*}
Now the zero divisor graphs induced by $\mathscr{A}(3)$ and $\mathscr{A}(5)$ are $\overline{K}_4$ and $\overline{K}_2,$ respectively. Hence $\Gamma(\mathbb{Z}_{15})=P_2[\overline{K}_4, \overline{K}_2].$ This graph can be visualized as follows:

\begin{center}
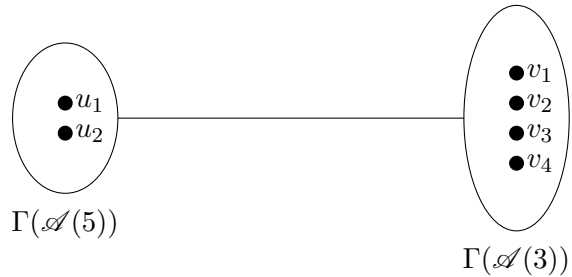

\begin{tikzpicture}
    % First ellipse
    \draw (0,0) ellipse (0.7cm and 1cm);
    \node at (0, -1.4) {$\Gamma(\mathscr{A}(5))$};

        \node at (0, -0.2) [circle, fill, inner sep=2pt] {};
        \node[right] at (0, -0.2) {$u_2$};
        
        \node at (0, 0.2) [circle, fill, inner sep=2pt] {};
        \node[right] at (0, 0.2) {$u_1$};

    \draw (6,0) ellipse (0.7cm and 1.5cm);
    \node at (6, -1.9) {$\Gamma(\mathscr{A}(3))$};

        \node at (6, -0.6) [circle, fill, inner sep=2pt] {};
        \node[right] at (6,-0.6) {$v_4$};
        \node at (6, -0.2) [circle, fill, inner sep=2pt] {};
        \node[right] at (6,-0.2) {$v_3$};
        \node at (6, 0.2) [circle, fill, inner sep=2pt] {};
        \node[right] at (6,0.2) {$v_2$};
        \node at (6, 0.6) [circle, fill, inner sep=2pt] {};
        \node[right] at (6,0.6) {$v_1$};

    \draw (0.7,0) -- (5.3,0);
\end{tikzpicture}
\captionsetup{labelfont={bf},textfont={it}}
\captionof{figure}{The generalized join of $\Gamma(\mathscr{A}(5))$ and $\Gamma(\mathscr{A}(3))$ which is isomorphic to $\mathbb{Z}_{15}.$}
\end{center}
Here, $u_1\sim v_1 \sim u_2\sim v_2\sim v_1$ is a circuit making the graph $\Gamma(\mathbb{Z}_{15})$ cyclic.
\end{example}

\bibliographystyle{unsrt}
\bibliography{./references}
\end{document}